\documentclass[11pt]{amsart} 
\usepackage{amssymb,amsmath,latexsym,enumerate,graphicx,bbm,mathptmx,microtype,cite}

\newtheorem{theorem}{Theorem} 
\newtheorem{corollary}[theorem]{Corollary}

\newtheorem{lemma}[theorem]{Lemma}

\newtheorem{exam}{Example}

\newtheorem*{rem}{Remarks}

\renewcommand\th{^{\text{th}}}

\newcommand\commentout[1]{}
\newcommand\Def[1]{{\bf #1}}

\newcommand\ZZ{\mathbb{Z}}

\newcommand\RR{\mathbb{R}}
\newcommand\CC{\mathbb{C}}

\newcommand\cP{\mathcal{P}}

\newcommand\bone{\mathbf{1}}

\newtheorem*{mainconjecture}{Antimagic Graph Conjecture}
\newtheorem*{newmainconjecture}{Directed Antimagic Graph Conjecture}

\begin{document}

\title{Partially Magic Labelings and the Antimagic Graph Conjecture}

\author{Matthias Beck}
\address{Department of Mathematics\\
         San Francisco State University\\
         San Francisco, CA 94132\\
         U.S.A.}
\email{mattbeck@sfsu.edu}

\author{Maryam Farahmand}
\address{Department of Mathematics\\
         University of California\\
         Berkeley, CA 94720\\
         U.S.A.}
\email{mfarahmand@math.berkeley.edu}


\begin{abstract}
The \emph{Antimagic Graph Conjecture} asserts that every connected graph $G = (V, E)$ except $K_2$ admits an edge
labeling such that each label $1, 2, \dots, |E|$ is used exactly once and the sums of the labels on all edges
incident to a given vertex are distinct.  On the other extreme, an  edge labeling is \emph {magic} if the sums of the
labels on all  edges incident to each vertex are the same. In this paper we approach antimagic labelings by
introducing \emph {partially magic labelings}, where ``magic occurs'' just in a subset of $V$.
We generalize Stanley's theorem about the magic graph labeling counting function to the associated
counting function of partially magic labelings and prove that it is  a quasi-polynomial of period at most $2$.
This allows us to introduce \emph{weak antimagic labelings} (for which repetition is allowed), and we show that
every bipartite graph satisfies a weakened version of the Antimagic Graph Conjecture.
\end{abstract}




\date{12 November 2015}

\thanks{The authors thank Thomas Zaslavsky for valuable comments.
This research was partially supported by the U.\ S.\ National Science Foundation (DMS-1162638).}

\maketitle


\section{Introduction}\label{intro}

Graph theory is abundant with fascinating open problems. In this paper we propose a new \emph{ansatz} to a
long-standing and still-wide-open conjecture, the \emph{Antimagic Graph Conjecture}.
Our approach generalizes Stanley's enumeration results for magic labelings of a graph \cite{Stan} to partially magic labelings, with which we analyze the structure of antimagic labelings of graphs.

Let $G$ be a finite graph, which may have loops and multiple edges. We shall denote the set of vertices of $G$  by
$V$ and the set of edges by $E$. 
A \Def{labeling} of $G$ is  an assignment $L: E \to \mathbb{Z}_{ \ge 0 }$ of a nonnegative integer $L(e)$ to each edge
$e$ of $G$ and a \Def{$k$-labeling} is one  where each edge label is among $0, 1, \dots, k$. If for every vertex $v$
of $G$ the sum $s(v)$ of the labels of all edges incident to $v$ equals $r$ (counting each loop at $v$ only once) then
$L$ is called a \Def{magic labeling of $G$ of index $r$}.
In the 1970s, Stanley proved some remarkable facts for magic labelings:

\begin{theorem}[Stanley \cite{Stan}]\label{1}
Let  $G$ be a  finite graph and define $H_G(r)$ to be the number of magic labelings of $G$ of index $r$. There exist
polynomials $P_G(r)$ and $Q_G(r)$ such that  $H_G(r)= P_G(r)+(-1)^r Q_G(r)$. Moreover, if the graph obtained by
removing all loops from $G$ is bipartite, then $Q_G(r)=0$, i.e., $H_G(r)$ is a polynomial of~$r$.
\end{theorem}

This theorem can be rephrased in the language of quasi-polynomials.
Recall that a \Def{quasi-polynomial} is a function $f: \ZZ \to \CC$ of the form $f(n)=c_{n}(k) \, k^{n} + \dots + c_{1}(k) \, k + c_{0}(k)$ where $c_0(k), \dots , c_n(k)$ are periodic functions in $k$ and the \Def{period} of $f$ is the least common multiple of the periods of $c_0(k), \dots , c_n(k)$. 
Theorem~\ref{1} says that $H_G(r)$ is a quasi-polynomial of period at most~$2$.

On the other extreme, a labeling is \Def{antimagic} if each edge label is a distinct element of  $\{1, 2, \dots , |E| \}$
so that the sums $s(v)$ are distinct. It has been con\-jectured for more than two decades that $K_2$ is essentially
the only graph for which we cannot find an antimagic labeling~\cite{Pearls}:

\begin{mainconjecture}
Every connected graph except for $K_2$ admits an antimagic labeling.
\end{mainconjecture}

Surprisingly this conjecture is still open for \Def{trees}, i.e., connected graphs without cycles, though it has been proven
that trees without vertices of degree $2$ admit an antimagic labeling \cite{Kaplanantimagic}. Moreover, the validity
of the Antimagic Graph Conjecture was proved in \cite{Alon} for connected graphs with minimum degree $\ge c \log |V|$
(where $c$ is a universal constant) and for connected graphs with maximum degree $\ge |V|-2$. We also know that
connected $k$-regular graphs with $k\ge 2$ are  antimagic \cite{Cranston,reg}. Furthermore, all Cartesian
products of regular graphs of positive degree are antimagic \cite{Cart}, as are joins of complete graphs \cite{join}.  For more related results, see the comprehensive survey \cite{Gallian} on graph labelings.

In the classic definition of  antimagic labelings, labels are distinct, however, for magic labelings repetition is
allowed. Borrowing a leaf from the latter, we soften the requirement in the antimagic definition above and say a
labeling is \Def{weakly antimagic} if each edge label is an element of $\{1, 2, \dots, |E|\}$ so that the sums $s(v)$ are
distinct. In other words, we allow repetition among the labels.
Our first main result is as follows.

\begin{theorem}\label{pml}
Let $G$ be a  finite graph. Then the number $A_G(k)$ of weakly antimagic $k$-labelings is a quasi-polynomial in $k$
of period at most $2$.  Moreover if the graph $G$ minus its loops is bipartite, then $A_G(k)$ is a polynomial in~$k$.
\end{theorem}

We remark that antimagic counting functions of the flavor of $A_G(k)$ already surfaced in \cite{iop}.
At any rate,  Theorem \ref{pml} implies that for the bipartite graphs we have a chance of using the
polynomial structure of $A_G(k)$ to say something about the antimagic character of $G$. Namely, the number
of $k$-labelings is at most  $(k+1)^{|E|}$ and so the degree of the polynomial $A_G(k)$ is at most $|E|$.
Therefore,  $A_G(k)$ can have at most $|E|$ integer roots, one of which is $0$, as we will show below. Also,
by definition we know that $A_G(k+1)\ge A_G(k)$ for any $k \in \ZZ_{\ge 0}$, and so $A_G(|E|)$ cannot be
zero.
What we just illustrated (though we will give a careful proof below) is that Theorem~\ref{pml} implies:

\begin{theorem}\label{pml2}
Every bipartite graph without a $K_2$ component admits a weakly antimagic labeling.
Furthermore, every graph $G=(V,E)$ without a $K_2$ component admits a labeling with distinct vertex sums $s(v)$ and labels in $\{1, 2, \dots, 2|E| \}$.
\end{theorem}

We approach weakly antimagic labelings by introducing a new twist on magic labelings which might be of independent
interest. Fix a subset $S$ of vertices of $G$.  A \Def{partially magic labeling of $G$ over $S$} is a labeling such
that ``magic occurs'' just in $S$, that is, for all $v \in S$ the sums $s(v)$ are equal.

\begin{theorem}\label{thm:partialmagiccounting} 
Let $G$ be a finite graph and $S \subseteq V.$ The number $M_S(k)$ of partially magic $k$-labelings over $S$ is a
quasi-polynomial in $k$ with period at most $2$. Moreover, if  the graph $G$ minus its loops is bipartite, then
$M_S(k)$ is a polynomial in~$k$.
\end{theorem} 

In order to prove Theorem~\ref{thm:partialmagiccounting}, we will follow Stanley's lead in \cite{Stan} and use linear
Diophantine homogeneous equation and Ehrhart quasi-polynomials to describe partially magic labelings of a graph;
Section~\ref{sec:partialmagic} contains a proof of Theorem~\ref{thm:partialmagiccounting}.
In Section~\ref{sec:antimagic} we prove Theorems~\ref{pml} and~\ref{pml2}.
We conclude in Section~\ref{sec:outlook} with some comments on a directed version of the Antimagic Graph Conjecture, as well as open problems.


\section{Enumerating Partially Magic Labeling}\label{sec:partialmagic}


Given a finite graph $G = (V, E)$ and a subset $S \subseteq V$, we introduce an indeterminate $z_e$ for each edge $e$
and let $\{v_1, \ldots, v_s\}$ be the set of all vertices of $S$, where $|S|=s$.
In this setup, a partially magic $k$-labeling over $S$ corresponds to an integer solution of the system of equations and inequalities

\begin{equation}\label{abc}
\sum_{e \text{ incident to } v_j} z_e \ = \sum_{e \text{ incident to } v_{j+1}} z_e \, \quad  (1 \le j \le s-1)
\quad \text{ and } \quad 0 \le z_e \le k \, .
\end{equation}
Define $\Phi$ as the set of all pairs $(L,k)$ where $L \in \ZZ_{\ge 0}^{E}$ is a partially magic $k$-labeling; that
is, $(L,k)$ is a solution to (\ref{abc}). 
If $L$ is a partially magic $k$-labeling and $L'$ is a partially magic $k'$-labeling, then $L+L'$ is a partially magic $(k+k')$-labeling.
Thus $\Phi$ is a semigroup with identity ${\bf 0}$. 
(This is also evident from~\eqref{abc}.)

For the next step, we will use the language of generating functions, encoding all partially magic $k$-labelings as
monomials. Let $q = |E|$ and define
\begin{equation}\label{a2}
F(Z) \ = \ 
F(z_1, \ldots, z_q, z_{q+1}) \ := 
\sum_{(L,k)\in \Phi} z_{1}^{L(e_1)}\cdots z_{q}^{L(e_q)} z_{q+1}^k \, .
\end{equation}
Note that if we substitute $z_{1}=\dots=z_{q}=1$ in $F(Z)$, we enumerate all partially magic $k$-labelings:
\begin{equation}\label{a3}
F(\bone, z_{ q+1 } ) \ = 
\sum_{(L,k) \in \Phi} z_{ q+1 }^k \ = \ 
\sum_{k \ge 0} M_S(k) \; z_{ q+1 }^k ,
\end{equation}
where we abbreviated $\bone := (1, 1, \dots, 1)$.

We call a nonzero element $\alpha=(\alpha_1, \ldots, \alpha_q,k) \in \Phi$ \Def{fundamental} if it cannot be
written as the sum of two nonzero elements of $\Phi$; furthermore, $\alpha$ is \Def{completely fundamental} if no positive
integer multiple of it can be written as the sum of nonzero, nonparallel elements of $\Phi$ (i.e., they are not
scaler multiple of each other). In other words, a completely fundamental element $\alpha \in \Phi$ is a nonnegative
integer vector such that for each positive integer $n$, if $n\alpha=\beta+\gamma$ for some $\beta, \gamma \in \Phi$,
then $\beta=j\alpha$ and $\gamma=(n-j) \alpha$ for some nonnegative integer $j$. Note that by taking $n=1$ in the
above definition, we see that every completely fundamental element is fundamental. Also note that any fundamental
element $(\alpha_1, \ldots, \alpha_q,k)$ necessarily satisfies $k=\max \{\alpha_1,\ldots,\alpha_q\}$.

Now we focus on the generating functions \eqref{a2} and \eqref{a3} and employ \cite[Theorem 4.5.11]{StanEC}, which
says in our case that 
the generation function $F(Z)$ can be written as a rational function with denominator 
\begin{equation}\label{6}
  D(Z) \ := \prod_{\alpha \in \operatorname{CF} (\Phi)} \left( 1- Z^{\alpha} \right) ,
\end{equation}
where $\operatorname{CF}(\Phi)$ is the set of completely fundamental elements of $\Phi$ and we used the 
monomial notation $Z^\alpha := z_1^{ \alpha_1 } z_2^{ \alpha_2 } \cdots z_q^{ \alpha_q } z_{ q+1 }^k$. 
To make use of \eqref{6}, we need to know some information about completely fundamental solutions to (\ref{abc}).
To this extent, we borrow the following lemmas from magic labelings \cite{Stan}, i.e., the case $S = V$: 

\begin{lemma}\label{3}
For a finite graph $G$, every completely fundamental magic labeling has index $1$ or $2$. More precisely, if $L$ is
any magic labeling of $G$, then $2L$ is a sum of magic labelings of index $2$.
\end{lemma}

\begin{lemma}\label{4} For a finite graph $G$, the following conditions are equivalent:
\begin{enumerate}[{\rm (1)}]
\item Every completely fundamental magic labeling of $G$ has index~$1$.
\item If $G'$ is any spanning subgraph of $G$ such that every connected component of $G'$ is a loop, an edge, or a
cycle of length $\ge 3$, then every one of these cycles of length greater than or equal to $3$ must have even length.
\end{enumerate}
\end{lemma}

Lemma \ref{3} implies that every completely fundamental magic labelings has index $1$ or $2$ and therefore, it cannot
have a label $\ge 3$ (because labels are nonnegative). By the same reasoning, if $G$ satisfies the condition $(2)$ in
Lemma \ref{4}, every completely fundamental magic labeling of it has index 1 and so cannot have labels $\ge 2$. We
now give the analogous result for \emph{partially} magic labelings:


\begin{lemma}\label{7}
Every completely fundamental partially magic labeling of $G$ over $S$ has labels $0$, $1$, or~$2$.
\end{lemma}

\begin{proof}
 If $S=V$, then every completely fundamental partially magic labeling over $S$ is a completely fundamental magic
labeling over $G$. By Lemma \ref{3}, it has index $1$ or $2$ and so the labels are among $0$, $1$, or~$2$.

Suppose that $S \subsetneq V$ and let $L$ be a partially magic labeling of $G$ over $S$ that has a label $\ge
3$ on  the edge $e$ which is incident to vertices $u$ and $v$. We will show that $L$ is not completely fundamental.
There are three cases:

\begin{figure}[htb]
\begin{center}
  \includegraphics [width=6in]{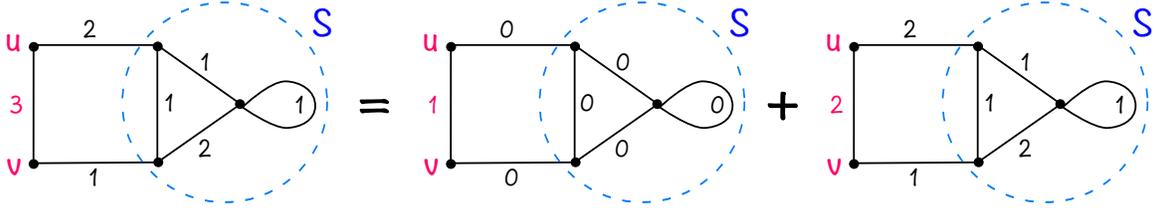}
\end{center}
\caption{A non-completely fundamental partially magic labeling in Case~1.}\label{F1}
\end{figure}

\noindent
{\it Case 1:}  $u, v \notin S$, that is, $e$ is not incident to any vertex in $S$.  We can write $L$
as the sum of $L'$ and $L''$, where all the labels of $L'$ are zero except for $e$ with $L'(e)=1$ and all the labels
of $L''$ are the same as $L$ except $e$ with $L''(e)=L(e)-1$; see Figure~\ref{F1}. Since $L'$ and $L''$ are both partially magic over $S$, then by definition $L$ is not a completely fundamental partially magic labeling over~$S$.

\vspace{10pt}
\noindent
{\it Case 2:} $u\notin S$ and $v \in S$. Let $G_S$ be the graph with vertex set $S$ obtained from $G$ by removing all
the edges of $G$ that are not incident to some vertex of $S$ and making loops out of those edges that are incident to
both $S$ and $V \setminus S$. Now define a labeling $L_S$ over $G_S$ such that all the edges that are incident to $S$
get the same labels as $L$ and all the new loops get the labels of $L$ that were on the original edges, as in Figure~\ref{F2}.

\begin{figure}[htb]
\begin{center}
  \includegraphics [width=4.3 in]{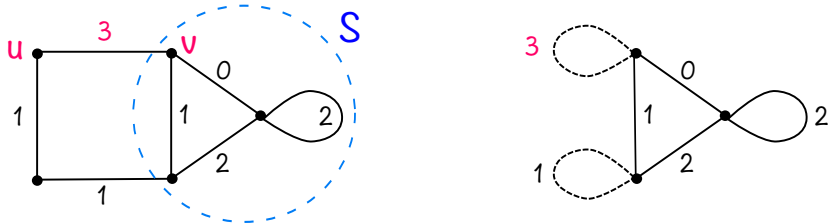}
\end{center}
\caption{A graph $G_S$ and magic labeling $L_S$ in Case~2.}\label{F2}
\end{figure}

Since $L$ is partially magic over $S$, $L_S$ is a magic labeling of $G_S$. However, $L_S(e)=L(e)\ge 3$ and so $S$ has
a vertex with sum $\ge 3$. Therefore, by Lemma \ref{3}, $L_S$ is not a completely fundamental magic labeling of $G_S$
and so there exist magic labelings $L_S^{i}$ of index $2$ such that $2 L_S=\sum L_{S}^{i}$, as in Figure~\ref{F3}. 
\begin{figure}[htb]
\begin{center}
  \includegraphics [width=6.5 in]{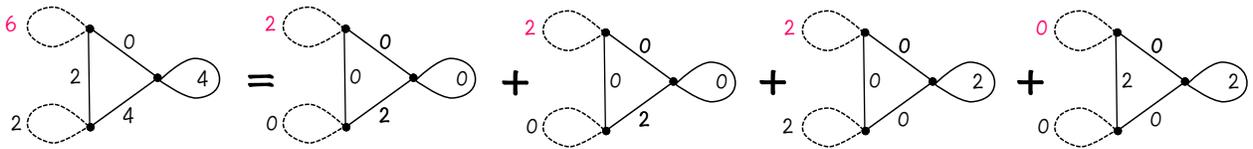}
\end{center}
\caption{A graph $G_S$ and the magic labelings $L_S^{i}$, where  $2L_S=\sum_{i=0}^{4} L_S^{i}$\;.}\label{F3}
\end{figure}
Now we extend each magic labeling $L_S^{i}$ to a partially magic labeling $L^{i}$ over $G$ as follows:
\[
  L^1(e) := \begin{cases}
            L^{1}_{S}(e) & \text{ if $e$ is incident to vertices of $S$ or $e$ is incident to  both vertices of $S$ and $V\setminus S$,}\\
            L(e) & \text{ if $e$ is not incident to $S$.}
            \end{cases} 
\]
Similarly we extend $L_{S}^{2}$ to $L^2$ on $G$. For $i \ge 3$, the extensions are:
\[
  L^i(e) := \begin{cases}
            L^{i}_{S}(e) & \text{ if $e$ is incident to vertices of $S$ or $e$ is incident to both vertices of $S$ and $V\setminus S$,}\\
            0 & \text{ if $e$ is not incident to $S$.}
            \end{cases}
\]
Therefore $ 2 L(e)=\sum L^{i}(e)$ for all $e \in E$; see Figure~\ref{F4}. 
By definition, $L^{i}$ is nonzero partially magic labeling of $G$ over $S$ with labels among $0,1,2$, for every $i>1$. This proves that $L$ is not completely fundamental partially magic labeling.

\begin{figure}[htb]
\begin{center}
  \includegraphics [width=6.5 in]{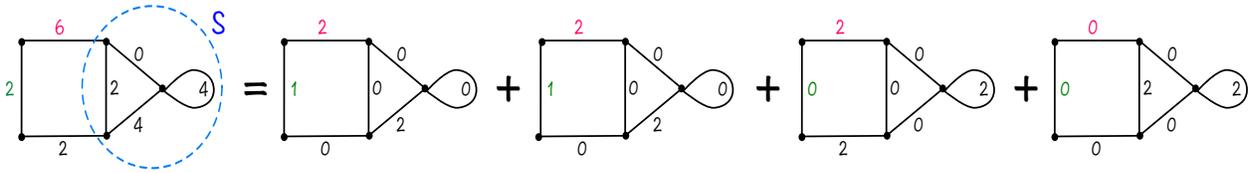}
\end{center}
\caption{A non-completely fundamental partially magic labeling $L$ with $2L=\sum_{i=0}^{4} L^{i}$\;.}\label{F4}
\end{figure}

\noindent
{\it Case 3:}  $u, v \in S$. The argument is similar to case 2, by constructing the graph $G_S$ with the labeling
$L_S$. Since $L$ is partially magic labeling over $S$, $L_S$ is a magic labeling of $G_S$. However, it is not
completely fundamental because it has a loop $e$ with label $L_S(e)=L(e)\ge 3$.  So there exist magic labelings
$L_S^{i}$ with labels among $0,1,2$, such that $2 L_S=\sum L_S^{i}$. We extend each $L_S^{i}$  to a labeling
$L^{i}$ over $G$ as follows:
\[
  L^1(e) := \begin{cases}
            L^{1}_{S}(e) & \text{ if $e$ is incident to vertices of $S$ or $e$ is incident to  both vertices of $S$ and $V\setminus S$,}\\
            L(e) & \text{ if $e$ is not incident to $S$.}
            \end{cases} 
\]
Similarly we extend $L_{S}^{2}$ to $L^2$ on $G$. For $i \ge 3$, we extend the labeling $L_{S}^{i}$ to $L^{i}$ over $G$
as follows:
\[
  L^i(e) := \begin{cases}
            L^{i}_{S}(e) & \text{ if $e$ is incident to vertices of $S$ or $e$ is incident to both vertices of $S$ and $V\setminus S$,}\\
            0 & \text{ if $e$ is not incident to $S$.}
            \end{cases}
\] 
By definition, $2L=\sum L^{i}$ where each $L^{i}$ is a partially magic labeling over $S$ and has labels $0,1$, or $2$.
Therefore, $L$ is not a completely fundamental magic labeling of $G$ over~$S$.
\end{proof}

\begin{proof}[Proof of Theorem \ref{thm:partialmagiccounting}]
By \eqref{a3} and \eqref{6},
\[
  F({\bf 1},z) \ = \ \sum_{k\ge 0} M_S(k) \, z^k
\]
is a rational function with denominator
\begin{equation}\label{M9}
  D({\bf 1},z) \ = \prod_{\beta \in \operatorname{CF} (\Phi)} \left( 1- {\bf 1}^{\beta}z^{k} \right)
\end{equation}
where $\operatorname{CF}(\Phi)$ is the set of completely fundamental elements of $\Phi$. 
According to Lemma \ref{7}, every completely fundamental element of $\Phi$ has labels at most $2$. Therefore
\begin{equation}\label{M0}
  \sum_{k\ge 0} M_S(k) \, z^k \ = \ \frac{h(z)}{(1-z)^{a}(1-z^2)^{b}}
\end{equation}
for some nonnegative integers $a$ and $b$, and some polynomial $h(z)$. 
Basic results on rational generating functions (see, e.g., \cite{StanEC}) imply that
$M_S(k)$ is a quasi-polynomial in $k$ with period at most 2.

Now let $G$ be a bipartite graph and $S \subseteq V$. We know that all the cycles of $G$ have even length, so it
satisfies  condition $(2)$ in Lemma \ref{4}. Therefore every completely fundamental magic labeling of $G$ has index
$1$ and so it cannot have a label $\ge 2$.
For partially magic labelings of $G$, we can use the same procedure of Lemma \ref{7} to see that if $L$ is a
completely fundamental partially magic labeling of $G$, it cannot have a label $\ge 2$. Therefore, in the generating
function (\ref{M0}), we have $b=0$ and so $M_S(k)$ is a polynomial in~$k$.
\end{proof}

The  equations in (\ref{abc}) together with $z_e \ge 0$ describe a pointed rational cone, and adding the
inequalities $z_e \le 1$  gives a rational polytope $\cP_S$.  Our reason for concentrating on the polytope
$\cP_S$ are structural results, due to Ehrhart and Macdonald (see, e.g., \cite{Ccd}), about the
lattice-point enumerator of any polytope $\cP \subset \RR^d$,
\[
  L_{\cP}(t) \ := \ \left| t \cP \cap \ZZ^{d} \right| ,
\]
where $t$ is a positive integer and $t \cP := \{tx : x \in \cP \}$ denotes the $t\th$ dilate of $\cP$. A partially magic $k$-labeling of a graph $G$ with labels among $\{0,1, \ldots, k\}$ (which is a solution of (\ref{abc})) is therefore an integer lattice point in the $k$-dilation of $\cP_S$, i.e., 
\[
  M_S(k) = L_{\cP_S}(k) \, .
\]
Let $M^{\circ}_S(k)$ be the number of \emph{positive} partially magic labelings of a graph $G$ over a subset $S$ of
vertices of $G$, that is, a partially magic labeling with labels among $\{1, \ldots, k\}$. Thus
${M}^{\circ}_S(k)=L_{\cP_S^{\circ}}(k+1)$, where $\cP^{\circ}_S$ is  the relative interior of the polytope
$\cP_S$. Ehrhart's famous theorem implies that $L_{\cP_S}(t)$ is a quasi-polynomial in $t$ of degree
$\dim{\cP_S}$, and the Ehrhart--Macdonald reciprocity theorem for rational polytopes gives the algebraic
relation $(-1)^{ \dim \cP } L_\cP(-t) = L_{ \cP^\circ } (t)$, which implies for us:

\begin{corollary}\label{12}
Let $G = (V, E)$ be a graph and $S \subseteq V$. Then $M^{\circ}_S(k)=\pm\;M_S(-k-1)$.
In particular, $M_s(k)$ and $M_S^\circ(k)$ are quasipolynomials with the same period.
\end{corollary}


\section{Antimagic Labelings}\label{sec:antimagic}

By definition of a partially magic labeling of a graph $G$ over a subset $S \subseteq V$, all the vertices
of $S$ have the same sum $s(v)$. If $S$ ranges over all subsets (of size $\ge 2$) of the vertices of $G$, we can write the number  $A_G(k)$ of  weak antimagic $k$-labelings as an inclusion-exclusion combination of the number of positive $k$-partially magic labelings:
\begin{equation}\label{11}
  A_G(k) \ = \sum_{\substack{S \subseteq V \\ |S| \ge 2}} c_S \, M^{\circ}_S(k)
\end{equation}
for some $c_S \in \ZZ$.
Thus Theorem \ref{thm:partialmagiccounting} and Corollary \ref{12} imply Theorem~\ref{pml}.

In preparation for our proof of Theorem~\ref{pml2}, we give a few basic properties of~$A_G(k)$.

\begin{lemma}\label{quasi}
The quasipolynomial $A_G(k)$ is constant zero if and only if $G$ has a $K_2$ component.
In either case, $A_G(0) = 0$.
\end{lemma}

\begin{proof}
If $G$ has a $K_2$ component, then clearly there is no antimagic labeling on $G$ and so $A_G(k)=0$. Conversely,  if $G$ is a graph with
minimum number of edges such that removing any arbitrary edge results in a $K_2$ component in $G$, then each component of $G$ is a path with
$3$ vertices and $2$ edges, which admits an antimagic edge labeling. 
Now assume that $G $ is a graph consisting of an
edge $e$ such that the graph  $G\setminus e$, obtained from $G$ by removing $e$, does not have any $K_2$ component. By induction,  $G
\setminus e$  admits an antimagic labeling.  Now for the graph $G$, we can label the edge $e = vu$ sufficiently large such that $s(v)$ and $s(u)$ are different from each other vertex sum. Thus $A_G(k)\neq 0$.

The second statement follows from \eqref{11}, since by definition $\cP_S \subseteq [0,1]^E$ and so $M_S^{\circ}(0)=L_{\cP_S^{\circ}} (1)=0$.
\end{proof}

\begin{proof}[Proof of Theorem~\ref{pml2}]
The second statement can be proven similarly to the first statement in Section \ref{intro}. By Theorem \ref{pml}, we know that $A_G(k)$ is a
quasi-polynomial in $k$ of period $\le 2$, and we also know that $A_G(k+1)\ge A_G(k)$. So both  even and odd constituents are  polynomials in
$k$ with degree at most $|E|$ and so they can have at most $|E|$ integer roots.  By Lemma \ref{quasi}, one of the roots is 0. Therefore $A_G(2|E|)>0$.
\end{proof}


\section{Concluding Remarks and Open Problems}\label{sec:outlook}

Among the more recent results on antimagic graphs are some for \Def{directed graphs} (for which one of
the endpoints of each edge $e$ is designated to be the head, the other the tail of $e$); given an edge
labeling of a directed graph, we denote the \Def{oriented sum} $s(v)$ at the vertex $v$ to be the sum of the labels of all edges oriented away from $v$ minus the sum of the labels of all edges oriented towards $v$.
Such a labeling is \Def{antimagic} if each label is a distinct element of $\{1, 2, \dots, |E| \}$ and the
oriented sums $s(v)$ are pairwise distinct. It is known that every directed graph whose underlying
undirected graph is dense (in the sense that the minimum degree at least $C \log |V|$ for some absolute
constant $C>0$) is antimagic, and that almost every regular graph admits an orientation that is
antimagic~\cite{OP}. Hefetz, M\"utze, and Schwartz suggest a directed version of the Antimagic Graph
Conjecture; the two natural exceptions are the complete graph $K_3$ on three vertices with an edge
orientation that makes an oriented cycle, and $K_{1,2}$, the bipartite graph on the vertex partition
$\{v_1\}$ and $\{v_2, v_3\}$ where the orientations are from $v_2$ to $v_1$ and $v_1$ to $v_3$.

\begin{newmainconjecture}
Every connected directed graph except for the directed graphs $K_3$ and $K_{1,2}$ admits an antimagic labeling.
\end{newmainconjecture}

It is tempting to adjust our techniques to the directed settings, but there seems to be road blocks.
For starters, no directed graph has a \Def{magic labeling}, i.e., all sums $s(v)$ are equal.
To see this, let $A$ be the square matrix with $A_{ij}$ the oriented  sum of the vertex $v_{i}$ using the labels of all  edges between ${v_i}$ and $v_{j}$. 
Now if $L$ is a magic labeling with index $r$, the sum of each row of $A$ equals $r$, and so $r$ is an
eigenvalue of $A$ (with eigenvector $[1,1,\dots,1]$). However, $A$ is by construction a skew matrix, and so
it cannot have a real eigenvalue.

At any rate, a directed graph will have \emph{partially} magic labelings, defined analogously to the
undirected graph, and so we can enumerate antimagic labelings according to the directed analogue of
\eqref{11}. To assert the existence of an antimagic labeling, one would like to bound the period of the
antimagic quasipolynomial, as in Theorem~\ref{thm:partialmagiccounting}. However, this does not seem
possible. Namely, if the subset $S \subset V$ includes a directed path
$\cdots \rightarrow v_1 \rightarrow v_2 \rightarrow \dots \rightarrow v_s \rightarrow \cdots$ such that
the vertices $v_2, \dots, v_{s-1}$ are not adjacent to any other vertices, then a completely fundamental
partially magic labeling $L_S$ with index $\ge 1$ implies that the label on each edge of the path is greater
than that on the previous one. Thus, contrary to the situation in Lemma \ref{7}, the upper bound for the
labels in $L_S$ can be arbitrarily large.
Consequently, the periods of the partial-magic quasi-polynomials, and thus those of the antimagic
quasi-polynomials, can be arbitrarily large.

The papers \cite{iop,OP} gives several further open problems on antimagic graphs, some of which could be
tackled with the methods presented here. 
%
We close with an open problem about a natural extension of our antimagic counting function.
Namely, it follows from \cite{iop} that the number of antimagic labelings of a given graph $G$ with
\emph{distinct} labels between 1 and $k$ is a quasi-polynomial in $k$. Can anything substantial be said
about its period? It is unclear to us whether the methods presented here are of any help, however, 
any positive result would open the door to applying these ideas once more towards the Antimagic Graph Conjecture.

\bibliographystyle{amsplain}
\bibliography{partialmagic}

\setlength{\parskip}{0cm} 

\end{document}